\documentclass[a4paper,12pt]{article}
\usepackage[T1]{fontenc}
\usepackage{amsmath,amssymb,amsthm,amsfonts,color,enumerate,textcomp,wrapfig,url,yfonts,fancyhdr,geometry}
\usepackage{stmaryrd}%pakiet do polecenia \bigtriangleup

\usepackage{bbding,enumitem,bbm}%do indykatora
\everymath{\displaystyle}%indeksy pod np max
\theoremstyle{plain} % ,,Styl'' twierdzeñ zwyk³y (wyt³uszczony tytu³,
\newtheorem{tw}{Theorem}[section]	%section oznacza numerowanie w ramach rozdzia³ów
\newtheorem{corollary}{Corollary}[section]

\theoremstyle{definition} 

\newtheorem{example}{Example}[section]
\newtheorem{cexample}{Counterexample}[section]
\newtheorem{remark}{Remark}[section]
\theoremstyle{remark}% styl ,,uwag'': tytu³ kursyw¹, treœæ pismem% prostym.
 \def\Xint#1{\mathchoice
 {\XXint\displaystyle\textstyle{#1}}%
 {\XXint\textstyle\scriptstyle{#1}}%
 {\XXint\scriptstyle\scriptscriptstyle{#1}}%
 {\XXint\scriptscriptstyle\scriptscriptstyle{#1}}%
 \!\int}
 \def\XXint#1#2#3{{\setbox0=\hbox{$#1{#2#3}{\int}$}
 \vcenter{\hbox{$#2#3$}}\kern-.5\wd0}}
 
 \def\dashint{\Xint-}
\newcommand\dashintl{\dashint\limits}

\newcommand\mR{{\mathbb R}}
\newcommand\moR{{\overline{\mathbb R}}}
\newcommand\mN{{\mathbb N}}

\newcommand\cA{{\mathcal A}} %sigma-cia³o
\newcommand\cF{{\mathcal F}} %sigma-cia³o

\newcommand\cM{{\mathcal M}} %sigma-cia³o
\newcommand\sP{\mathsf{P}}

\newcommand\ve{\varepsilon}

\newcommand\mmm{\big[\inf H\circ \mu(A)\big]}

\renewcommand{\c}{\circ}

\newcommand\st{\star}%conjunction

\newcommand{\rS}{\mathrm{S}} %semikopu³a
\newcommand\intl{\int\limits}

\newcommand\rIl[3][\mu]{\dashintl_{#2}#3 \md #1}%oznaczenie na ca³kê Sugeno
\newcommand\rI[3][\mu]{\textstyle{\dashint_{#2}#3 \md #1}}%oznaczenie na ca³kê Sugeno do tekstu

\newcommand\gIl[3][\mu]{\intl_{#2}#3 \md #1}%oznaczenie na uogólnion¹ ca³kê Sugeno
\newcommand\gI[3][\mu]{\textstyle{\int_{#2}#3 \md #1}}
\newcommand\rrI[2][\mu]{\mathrm{Su}^\star_{#1,A}(#2)}%oznaczenie na ca³kê Sugeno z funkcji rzeczywistej z gwiazdk¹

\newcommand\md{\,{\mathrm{d}}}
\renewcommand\ge{\geqslant}
\renewcommand\le{\leqslant}

\newcommand{\mI}[1]{\mathbbm{1}_{#1}}

\title{Sharp  bounds of Jensen type\\ for the generalized Sugeno integral}
\author{Micha\l Boczek\footnote{Corresponding author. E-mail adress: michal.boczek.1@p.lodz.pl;}, Marek Kaluszka\\{\emph{
\small{Institute of Mathematics, Lodz University of Technology, 90-924 Lodz, Poland}}}}
\date{}

\begin{document}
\maketitle

\begin{abstract}
In this paper we provide   two-sided  attainable  bounds of Jensen type for 
the generalized Sugeno integral of {\it any} measurable function.
The results extend the previous results of Rom\'an-Flores et al. for increasing functions and Abbaszadeh et al. for convex and concave functions. We also give  corrections of some  results of Abbaszadeh et al. 
As a~by-product, we obtain sharp inequalities for 
symmetric integral of Grabisch.
To the best of our knowledge,  
the  results in the real-valued functions  context are presented for the first time here. 
\end{abstract}

%\begin{keyword}

%\end{frontmatter}
\section{Introduction} 
Let $(X,\cA)$ be a~measurable space, where $\cA$ is a~$\sigma$-algebra of subsets of a~nonempty set $X.$
A~{\it monotone measure} on $\cA$  is a~nondecreasing set function $\mu\colon \cA\to \moR_+$, i.e. $\mu(A)\le\mu(B)$ whenever $A\subset B$ with $\mu(\emptyset)=0,$ where $\moR_+=[0,\infty].$ 
We denote the range of $\mu$ by $\mu(\cA)$ and the class of all monotone measures  on $(X,\cA)$  by $\cM_{(X,\cA)}.$    
The class of all $\cA$-measurable functions $f\colon X\to Y$
is denoted by $\mathcal F_{(X,Y)},$ where $Y\subset \moR_+.$ A~binary map   $\c\colon \moR_+\times \moR_+\to \moR_+$ is said 
to be nondecreasing  if $a\c b\le c\c d$ for all
$a\le c$ and  $b\le d$.
% $a\le b$ and $c\le d.$ 
The \textit{generalized Sugeno integral} of $f\in\cF_{(X,\moR_+)}$ on $A\in\cA$  is defined as 
\begin{align}\label{gensugeno}
\gIl{\c, A}{f}:=\sup_{t\ge 0}\big\{t\c \mu(A\cap \{ f\ge t\} )\big\},
\end{align}
where $\{ f\ge t\}=\{x\in X\colon f(x)\ge t\},$ $\mu$ is a~monotone measure on $\cA$  and $\c$ is a~nondecreasing binary map. 
Commonly encountered examples of the generalized Sugeno integral include
the 
\textit{Sugeno integral} \cite{sug}  
\begin{align}\label{sugeno}
\rIl{A}{f}=\sup_{t\ge 0}\big\{t\wedge \mu(A\cap \{ f\ge t\} )\big\},
\end{align}
the \textit{Shilkret integral} \cite{shilkret}
\begin{align*}
\gIl{\cdot,A}{f}=\sup_{t\ge 0}\big\{t\cdot \mu(A\cap \{ f\ge t\})\big\},
\end{align*}
the \textit{q-integral} \cite{dubois2, dubois3} and  
the \textit{seminormed fuzzy integral} \cite{BHH, suarez}. 
Here and subsequently,  $a\wedge b=\min(a,b)$ and $a\vee b=\max(a,b).$

One of the most important  
inequalities 
in mathematics, economics and information theory is the Jensen inequality. The classical integral form of 
Jensen inequality states that
\begin{align}
\int _XH(f(x))\,\sP(\mathrm{d}x)\ge H\Big(\int _Xf(x)\,\sP(\mathrm{d}x)\Big),
\end{align}
where $(X,\cA,\sP)$ is a probability measure space, $H$ is a real-valued convex function on an interval $I$
of the real line, $f(x)\in I$ for all $x$ and $f\in L^1(\sP)$. 
Numerous
applications of the Jensen inequality are presented in \cite{ni,pe}. 

Given  a~function $H$ and  nondecreasing binary operations $\circ,\star$, we say that a~\textit{lower Jensen type bound} holds for the generalized Sugeno integral if there exists a function $\widehat{H}$ such that  for any $f\in \mathcal F_{(X,Y)}$ 
\begin{align}\label{nnn1}
\gIl{\c,A}{H(f)}\ge \widehat{H}\Big(\gIl{\star,A}{f}\Big).
%,
\end{align}
%where $\c$ and $\star$ are  nondecreasing binary maps. 

Replacing  ``$\ge$'' with   ``$\le$'' in \eqref{nnn1} gives an \textit{upper Jensen type bound}. The study of 
Jensen type inequalities for the Sugeno integral was initiated by Román-Flores and Chalco-Cano \cite{flores3}.
They provided  bounds  for strictly monotone nonnegative real functions and continuous monotone measure.
Since then, the fuzzy integral counterparts of the Jensen  inequality  have been studied   by Caballero and Sadarangani \cite{caballero3}, Daraby and  Rahimi \cite{da}, 
as well as Jaddi et al. \cite{Jaddi}. Kaluszka 
et al.  \cite{boczek2} presented necessary and sufficient conditions for  
the validity  of Jensen type inequalities for the generalized Sugeno integrals under monotonicity condition.

% In the recent article Jaddi et al. [1] claim to prove an equivalent condition for the classical Jensen inequality of convex functions in the framework of the generalized Sugeno integral. However, the result 
As one of the referees pointed out, the result of Theorem 3.1 in \cite{Jaddi} is a~special case of Theorem 2.3 in \cite{boczek2}. Indeed, if $H \colon Y\to Y$ is a~differentiable convex function with $H'(y)\ge 1$ for each $y\in Y,$ then it is nondecreasing and
left-continuous on $Y,$ thus the result (necessary and sufficient condition for the Jensen integral inequality) follows from \cite[Theorem 2.3]{boczek2}. Corollaries 3.2, 3.3 and
3.4 from \cite{Jaddi} are in fact the results from \cite{boczek2} after Theorem 2.3 therein. Moreover,
the assumption on continuity of monotone measure $\mu$ is also a~superfluous
constraint in \cite{Jaddi}. Also, the result for the discrete case (Theorem 4.1 in \cite{Jaddi}) is
immediate.

 Abbaszadeh  et al.~\cite{abba} obtained new Jensen type inequalities using 
concavity/convexity of  $H,$ 
but some of these results  are not valid (see counterexamples  below).
Generalizations of 
Jensen integral inequality for the pseudo-integral are proven by Pap and \v Strboja~\cite{pap6}.
Agahi et al. \cite{aga} extended the Jensen type inequality 
on 
$g$-expectation with general kernels.
Costa~\cite{costa}  provided fuzzy versions of Jensen inequalities type integral for convex and concave fuzzy-interval-valued functions.
 
In this article, we use a~new method of proof to establish     
some Jensen type inequalities for the generalized Sugeno integral of {\it any} measurable function $H.$  
We also improve and correct the Jensen type inequalities 
for the  Sugeno integral previously proposed in the literature. 
Moreover, we give  the Jensen type bounds for the symmetric Sugeno integral introduced by Grabisch \cite{gra0}, which have not been considered in the literature so far.

 The paper is organized as follows.
In Section 2, we derive sharp lower and upper bounds for the generalized Sugeno integral and  nonnegative function $H$ without the assumptions of convexity,  concavity or monotonicity of $H$. In Section 3, we 
deduce some  
Jensen type bounds 
from  
a~Liapunov type inequality for nonnegative concave functions.
Our final section 
provides a~Jensen type inequality for  the $\star$-symmetric Sugeno integral 
having both upper and lower estimates. 

\section{Jensen type  bounds for
nonnegative functions} 

We say that a~monotone measure $\mu$ is \textit{weakly subadditive} on $A\in \cA,$ if $\mu(A)\le  \mu(A\cap B)+\mu(A\cap B^c)$ for all $B,$ where $B^c=X\backslash B$.  
A~measure $\mu$ is \textit{weakly superadditive} on $A,$  if  ``$\le$'' is replaced by ``$\ge $'' in the definition of weak subadditivity on $A$. Clearly, any subadditive measure is weakly subadditive on any measurable set $A,$ but  a~weakly subadditive measure 
 need not
%does not need to
 be subadditive. 
For example,  the monotone measure $\mu$ on $X=\{1,2,3\}$ defined by  $\mu(\{1,2,3\})=2,$ $\mu(\{k\})=0.5$ and $\mu(\{k,l\})=1$ for all $k,l,$ is not subadditive while  it is weakly subadditive on $A=\{1,2\}.$

Throughout the paper, $\inf H(A)=\inf _{y\in A}H(y),$ $\sup H(A)=\sup _{y\in A}H(y)$,
$\inf H=\inf H(\moR_+),$ and $\sup H=\sup H(\moR_+)$ for any function $H\colon \moR_+\to \moR_+$ and $A\subset \moR_+$.   
Denote by $H(p_-)$ and $H(p_+)$ the lower left-hand  limit and the lower right-hand  limit of 
$H$ at $p,$ respectively, that is, 
$
H(p_-)=\lim_{\ve\to 0} \inf H((p-\ve,p))$ and $H(p_+)=\lim_{\ve\to 0} \inf H((p,p+\ve)).
$  
Hereafter, $H(0_-)=0$.

First, we give lower bounds of Jensen type. 
\begin{tw}\label{tw1}
Suppose that  $\c\colon \moR_+\times \moR_+\to \moR_+$ is a~nondecreasing map such that $a\c 0=0$ for all $a$ and  $x\mapsto x\c y$  is a~left-continuous function for any fixed $y$. 
Suppose also that $f,H(f)\in\cF_{(X,\moR_+)}$ and $p=\rI{A}{f}<\infty.$   
 \begin{enumerate}
 \item[(i)] The following inequality holds 
\begin{align}\label{in1}
\gIl{\c,A}{H(f)}\ge \big[\big(H(p_-)\wedge \inf H([p,\infty])\big)\c p\big] \vee \mmm.
\end{align}
 There is equality in  \eqref{in1}  for $f=\mu(A)\mI{A}$ if $H$ is left-continuous at $p$
and $H(p)=\inf H([p,\infty])$.
\item[(ii)] If  
$\mu$ is weakly subadditive on $A$, 
then 
\begin{align}\label{in2}
\gIl{\c,A}{H(f)}\ge \big[\big(H(p_+)\wedge \inf H([0,p])\big)\c (\mu(A)-p)\big]\vee \mmm.
\end{align} 
The equality holds in \eqref{in2} if
$f=y_0\mI{A},$   $H$ is 
right-continuous at $p$, $H(p)=\inf H([0,p])$ and 
$H(y_0)=\inf H$ for some $y_0$.
\end{enumerate}
\end{tw}
\begin{proof} 
(i) Assume that $p>0,$ as the bound \eqref{in1} is trivial for $p=0$.    
Let  $h(\ve) =\inf H([p-\ve,\infty])$ for $\ve\in (0,p).$ Define  $H_0(s)=\inf H$ for $s<p-\ve$ and $H_0(s)=h(\ve)$ for $s\ge p-\ve$. Clearly, $H(s)\ge H_0(s)$ for all $s\in \moR_+.$
Thus,  we have from the monotonicity of the generalized Sugeno integral that 
\begin{align*}
\gIl{\c,A}{H(f)}&\ge \gIl{\c,A}{H_0(f)}=\sup_{0\le t\le \inf  H}\left\{t\c \mu(A)\right\}\vee \sup _{t>\inf H}\left\{ t\c \mu(A\cap \{H_0(f)\ge t\} )\right\}
\\&= \mmm\vee  \big[ h(\ve)\c \mu(A\cap \{H_0(f)\ge h(\ve)\})\big]
\\&=\mmm\vee  \big[h(\ve)\c \mu(A\cap \{f\ge p-\ve\} )\big].
\end{align*}
It is well known that   $\mu(A\cap\{ f\ge y\})\ge p$ for all $y<p,$ where $p=\rI{A}{f}$ (see \cite[Lemma 9.7]{wang}). Therefore, 
\begin{align*}
\gIl{\c,A}{H(f)}
&\ge \mmm\vee \big[h(\ve)\c p\big].
\end{align*}
By left-continuity  of 
$x\mapsto x\c p$ and monotonicity of $h(\ve),$ we obtain
\begin{align}\label{no1}
 \gIl{\c,A}{H(f)}&\ge \mmm\vee \lim _{\ve\to 0} \big[h(\ve)\c p\big]\nonumber\\
&=\mmm\vee \big[\lim _{\ve \to 0}h(\ve)\c p\big] \nonumber \\&=\mmm\vee \big[\big(H(p_-)\wedge \inf H([p,\infty])\big)\c p\big].
 \end{align}  
Equality holds in \eqref{no1} for   $f=\mu(A)\mI{A}$  provided 
that $H$ is left-continuous at $p$ and $H(p)=\inf H([p,\infty])$. 

\noindent (ii)  Let $h(\ve)=\inf H([0,p+\ve])$ for all  
$\ve>0$. 
Put $H_0(s)=\inf H$ if $s>p+\ve$ and $H_0(s)=h(\ve)$ if $s\le p+\ve$. 
Weak subadditivity of $\mu$ implies  that
\begin{align*}
\gIl{\c,A}{H(f)}&\ge \gIl{\c,A}{H_0(f)}= \mmm\vee \sup _{t> \inf H}
\left\{ t\c \mu(A\cap \{H_0(f)\ge t\} )\right\}\\
&=\mmm\vee 
\big[h(\ve)\c \mu(A\cap \{f\le p+\ve\} )\big]\\
&\ge  \mmm\vee 
\big[h(\ve)\c (\mu(A)-\mu (A\cap\{f>p+\ve\}))\big].
 \end{align*} 
It follows from \cite[Lemma 9.7]{wang} that $\mu(A\cap\{ f>y\} )\le p<\infty$ for all $y>p.$ 
By the  monotonicity of $h(\ve)$ and left-continuity of map $y\mapsto y\c (\mu(A)-p)$, we get  
\begin{align}\label{no3}
\gIl{\c,A}{H(f)}&\ge  \mmm\vee \lim_{\ve\to 0}\big[h(\ve)\c (\mu(A)-p)\big]\nonumber\\
&=\mmm\vee  \big[(H(p_+)\wedge \inf H([0,p]))\c (\mu(A)-p)\big].
\end{align} 
There is  equality in \eqref{no3} 
for  
$f=y_0\mI{A}$, 
if  $H$ is right-continuous at $p,$ $H(y_0)=\inf H$   and $H(p)=\inf H([0,p]).$ Here and subsequently, $\infty\cdot 0=0.$ 
\end{proof}

\begin{remark} The bound \eqref{in1} (resp. \eqref{in2}) is sharp for each $p,$
if function $H$ is nondecreasing and left-continuous (resp. nonincreasing and right-continuous).
Moreover, if  $\mu$ is  a~subadditive  monotone measure and $H$ is 
a~continuous function, then $H(p_-)=H(p_+)=H(p)$ and  we have from \eqref{in1} and 
\eqref{in2} that 
\begin{align}\label{noo1}
\gIl{\c,A}{H(f)}\ge \big[\inf H([p,\infty])\c p\big] \vee
\big[\inf H([0,p])\c (\mu(A)-p)\big]\vee \mmm.
\end{align}
Assume  additionally that $H$ is quasiconvex, that is, $H$ is nonincreasing on 
$[0,a]$ and nondecreasing  on $[a,\infty]$ for some $a\in (0,\infty)$ \cite[\hbox{p. 99}]{boyd}. Then the bound \eqref{noo1} is attainable for every $p$ as
$H(p)=\inf H([p,\infty])$ or $H(p)=\inf H([0,p]).$ 
\end{remark}

Now we provide some consequences of  Theorem~\ref{tw1} for 
the Sugeno integral.

\begin{corollary} \label{co0} Assume that $H\colon \moR_+\to \moR_+$ is nondecreasing and left-continuous at $p$, where  
$p=\rI{A}{f}<\infty$ and $f\in\cF_{(X,\moR_+)}.$
Then the following sharp bound holds  
\begin{align}\label{flo}
\rIl{A}{H(f)}\ge H(p)\wedge p.
\end{align}  
\end{corollary}
\begin{proof} Apply  Theorem \ref{tw1}\,(i) with $\circ=\wedge.$ 
Inequality \eqref{flo} is attainable if  $f=\mu(A)\mI{A}$.
\end{proof}

Corollary \ref{co0} generalizes   
Corollary $3.3$ of  \cite{agahi16},   
Lemma 1 of \cite{caballero3} and {\cite[Theorem $2.1$]{boczek2}} for  $\c=\st=\wedge$ and $Y=\moR_+.$

The following example shows that the equality in \eqref{flo} may be 
achieved by  a~nonconstant function $f.$

\begin{example}\label{ex3} 
Let $X=\mR_+$, $A=\{1,2,3,4,5\}$ and $\mu$  
be the counting measure on $\mR_+,$ which means that $\mu(B)=\infty$, if $B$ is an infinite subset of $\mR_+$ and $\mu(B)= \textrm{card} (B),$
% the number of elements $B$, 
if $B$ is a~finite subset of $\mR_+.$ Take $H(x)=x^2/3$ and  $f(x)=x$.  Then
\begin{align*}
\rIl{A}{f}&=\sup _{t\ge 0}
\left\{t\wedge \mu(A\cap \{ f\ge t\})\right\}=
\max_{i\in A}\{i\wedge (6-i)\}=3,\\ 
\rIl{A}{H(f)}&=\max_{i\in A}\left\{(i^2/3)\wedge (6-i)\right\}=3,
\end{align*}
so the bound  \eqref{flo} is reached if $f(x)=x$. 
\end{example}
\medskip

The next Corollary  is a~corrected version of  Theorem 5.1  in
%of~
\cite{abba}. 
    
\begin{corollary} \label{co1} Suppose that $H\colon \mR_+\to 
\mR_+$ is a~convex function which attains its infimum at point $a.$  The 
%following 
sharp inequality 
\begin{align}\label{in2a}
\rIl{A}{H(f)}\ge H(p)\wedge p
\end{align} 
 holds for any  $f\in \cF_{(X,\mR_+)}$  
such that  
$p=\rI{A}{f}\in [a,\infty).$
\end{corollary}

\noindent It is easy to check that Theorem 5.1  in
%of
 \cite{abba} is 
not true without the additional 
assumption that $\varphi '(p)=1,$
but this assumption implies that $\varphi(p)\le \varphi'(p)p=p,$ where we follow the notation  in
%of
 \cite{abba}.  
\medskip

 \begin{cexample}\label{cex1}
Let us consider the space $([0,5],\cA,\mu)$ with the Lebesgue measure $\mu.$ Take $\varphi(x)=(x-0.5)^2$ and $f(x)=x.$ Clearly, $\varphi$ is 
a~differentiable convex function and $\varphi(x)\le x\varphi'(x)$ for  $x\in [0,5].$ All assumptions of Theorem 5.1 from \cite{abba} are satisfied.  It is easy to check that  $\rI{[0,5]}{f}=\sup_{t\ge 0}\left\{ t\wedge (5-t)\right\}=2.5$ and
\begin{align*}
\rIl[\mu]{[0,5]}{\varphi(f)}&=\sup_{t\in [0,\,0.25)}\big\{ t\wedge (5-2\sqrt{t})\big\}\vee\sup_{t\in [0.25,\,4.5^2]}\big\{ t\wedge (4.5-\sqrt{t})\big\}\\&=\frac{10-\sqrt{19}}{2}\in (2.82,2.821). 
\end{align*}
Thus, $\varphi\big(\rI{[0,5]}{f}\big)=4>\rI{[0,5]}{\varphi(f)}.$ Note that $\varphi'\big(\rI{[0,5]}{f}\big)=4\neq 1,$ but \eqref{in2a} holds, i.e.
\begin{align*}
\rIl{A}{\varphi(f)}\ge \varphi(p)\wedge p=2.5.
\end{align*}
\end{cexample}
 
Next we provide a~lower bound of Jensen type for the Shilkret integral.
\begin{corollary}\label{cono} 
If  $H\colon \mR_+\to \mR_+$ is nondecreasing left-continuous at $p=\rI{A}{f},$ 
where 
$f\in\cF_{(X,\mR_+)}$ and $\mu(A)<\infty,$
then  the following attainable bound is valid
\begin{align}\label{pp1}
\gIl{\cdot,A}{H(f)}\ge H(p)p.
\end{align}  
\end{corollary}
\begin{proof}
Setting $\circ=\cdot$ in Theorem~\ref{tw1}\,(i),  
we get
\begin{align*}
\gIl{\cdot,A}{H(f)}\ge  \big[(H(p)\wedge H(p)) \cdot p\big]\vee \big[H(0)\cdot \mu(A)\big]\ge H(p)p.
\end{align*}
\end{proof}

\begin{example}\label{ex4}
Let $X=[0,1]$ and $\mu=\lambda^q,$ where $\lambda$ is the Lebesgue measure and $q>0$.  Take $H(x)=x^{1/q}$ and  $f(x)=x^q$.  Then
\begin{align*}
\rIl{X}{f}&=\sup_{0\le t\le 1}\{ t\wedge (1-t^{1/q})^q\}=0.5^q,\\
\gIl{\cdot,X}{H(f)}&=\sup_{0\le t\le 1}\{ t\cdot (1-t)^q\}=\frac{1}{q}\left(\frac{q}{1+q}\right)^{q+1}.
\end{align*}
We  get from \eqref{pp1} that
$\gI{\cdot,X}{H(f)}\ge 0.5^{q+1}.$
\end{example}

\begin{example}\label{ex5} Using Corollary \ref{cono} for $H(x)=ax$, $a>0,$ we obtain the following inequality for the Shilkret integral 
\begin{align*}
\gIl{\cdot,A}{f}\ge \Big(\rIl{A}{f}\Big)^2,
\end{align*}
where $f\in\cF_{(X,\mR_+)}.$ This bound is  obvious (see the geometric interpretation of the Sugeno integral and the Shilkret integral), but  it shows that   the equality  in \eqref{in1} may hold
%may hold in \eqref{in1} 
not only for piecewise constant functions $f$ when $\c=\cdot$.
\end{example} 
\medskip

Recall that a~nondecreasing map $\otimes \colon [0,1]^2\to [0,1]$ is said to be a~\textit{fuzzy conjunction}
 if $1\otimes  1=1$ and $0\otimes  1=1\otimes  0=0\otimes  0=0$  (see \cite[Definition 2]{dubois3}). 
 The special case of 
 the
  fuzzy conjunction is  a~\textit{semicopula}, which has extra limit conditions
%, i.e.
 $a\otimes  1=1\otimes  a=a$ 
 (cf.~\cite{gra1}).
Dubois et al. \cite{dubois2} introduced and studied the $q$-integral defined as
\begin{align}
\int_\mu^\otimes f=\sup_{t\in [0,1]}\big\{\mu (\{f\ge t\})\otimes t\big\},
\end{align} 
where $\otimes$ denotes a~fuzzy conjunction, $f\in\cF_{(X,[0,1])}$ and $X$ is a~finite set (see also~\cite{dubois3}). This  
definition is motivated by alternative ways of using weights of qualitative criteria in min- and max-based aggregations, that make intuitive sense as tolerance thresholds.
In the literature, we can find  Jensen type bounds for  q-integral if $H$  is a~nondecreasing function (see~\cite[Theorems 2.1-2.3 and  Theorem 3.3]{boczek2}). Now we give  
their counterparts for a~quasiconvex function $H$. 

 \begin{corollary}\label{coq} Assume that  a~fuzzy conjunction $\otimes$ is  left-continuous in the second coordinate,  
 $\mu(\cA)\subset [0,1]$ and  $H\colon [0,1]\to [0,1]$ is a~quasiconvex function which attains its infimum at  point $a_0.$ 
Then, for all $f\in \cF_{(X,[0,1])}$  such that  $p=\textstyle{\int _\mu^\otimes f}\in [a_0,1]$, 
we have 
\begin{align*}
\int _\mu^\otimes H(f)\ge p\otimes H(p_-).
\end{align*}
\end{corollary}
\begin{proof} Put $a\c b=(b\wedge 1)\otimes (a\wedge 1)$ 
in \eqref{gensugeno}.
Note that 
$a\c 0=0$ as $0\le 0\otimes a\le 0\otimes 1=0$ for all $a$. Moreover,
\begin{align*}
\int _\mu^\otimes H(f)=\sup_{t\in [0,1]}\left\{t\c \mu (\{H(f)\ge t\})\right\}=
\sup_{t\ge 0}\left\{t\c \mu (\{H(f)\ge t\})\right\}=\gIl{\c,X}{H(f)}.
\end{align*} 
The assertion follows from Theorem \ref{tw1}\,(i). 
\end{proof}

Applying Theorem \ref{tw1} one can also obtain lower bounds by means of $\gI{\c,A}{f}$ instead of the Sugeno integral.

\begin{corollary}
Assume that   a~semicopula $\rS\colon [0,1]^2\to [0,1]$ is left-continuous in the first coordinate, $\mu(\cA)\subset [0,1]$  
 and  $H\colon [0,1]\to [0,1]$  is a~left-continuous and nondecreasing function on $[a_0,1]$ for some $a_0\in [0,1].$ 
Then  the following  sharp inequality for the seminormed fuzzy integral holds for all $f\in \cF_{(X,[0,1])}$  %such that  $p_S:=\gI{\rS,A}{f}\in [a_0,1]$, 
\begin{align*}
\gIl{\rS,A}{H(f)}\ge \rS(H(p_S),p_S),
\end{align*}
where $p_S:=\gI{\rS,A}{f}\in [a_0,1].$
\end{corollary}
\begin{proof} Take $a\c b=\rS (a\wedge 1,b\wedge 1)$ in  \eqref{gensugeno}. 
It is clear that   
$$\gIl{\c,A}{f}=\gIl{\rS,A}{f}=\sup_{t\in [0,1]} \rS\big(t,\mu(A\cap\{ f\ge t\})\big).$$
As  $\rS(a,b)\le a\wedge b$ for all $a,b$, we have 
$a_0\le p_S\le \rI{A}{f}\le 1.$
Moreover, from Theorem \ref{tw1}\,(i) and monotonicity of $H$ on $[a_0,1]$ we get
\begin{align*}
\gIl{\rS,A}{H(f)}\ge \rS\Big(H\Big(\rIl{A}{f}\Big),\rIl{A}{f}\Big)\ge   \rS(H(p_S),p_S).
\end{align*} 
This bound is reached for $f=\mu(A)\mI{A}$ if 
$\rS(\mu(A), \mu(A))=\mu(A)$. 
\end{proof}

Next we find 
some upper bounds of Jensen type. Let $H\colon \moR_+\to \moR_+$ be 
%any
 a~Borel measurable function. 
Denote by $H(p^-)$ and $H(p^+)$ the upper left-hand  limit and the upper right-hand  limit of 
$H$ at $p,$ respectively, that is, 
$
H(p^-)=\lim_{\ve\to 0} \sup H((p-\ve,p))$ with $H(0^-)=0,$ and $H(p^+)=\lim_{\ve\to 0}\sup H((p,p+\ve)).$

\begin{tw}\label{tw2}   
Let $\c\colon \moR_+\times \moR_+\to \moR_+$ be a~nondecreasing map such  that 
$x\mapsto x\c y$ is right-continuous for any fixed $y$ and $a\c 0=0$ for all $a.$ 
Assume that $f,H(f)\in\cF_{(X,\moR_+)}$ and $p=\rI{A}{f}<\infty.$
 \begin{enumerate}
 \item[(i)] The following bound is valid  
\begin{align}\label{in3}
\gIl{\c,A}{H(f)}\le \big[\big(H(p^+)\vee \sup H([0,p])\big)\c \mu(A)\big] \vee \big[\sup H\c p\big].
\end{align}
The equality holds in \eqref{in3} for $f=y_0\mI{A}$ if $H$ is right-continuous at $p$, $H(p)=\sup H([0,p])$ and $H(y_0)=\sup H$ for some $y_0$. 
\item[(ii)] If 
$\mu$ is weakly superadditive on $A$, then 
\begin{align}\label{in4}
\gIl{\c,A}{H(f)}\le \big[\big(H(p^-)\vee \sup H([p,\infty])\big)\c \mu(A)\big] \vee \big[\sup H\c (\mu(A)-p)\big]. 
\end{align} 
The equality in~\eqref{in4} is 
attained  for $f=\mu(A)\mI{A}$
if  $H$ is left-continuous at $p$ and $H(p)=\sup H([p,\infty])$. 
\end{enumerate}
\end{tw}
\begin{proof} 
(i)  Let $h(\ve)=\sup H([0,p+\ve])$ for all $\ve>0.$ 
Put $H_0(s)=\sup H$ for $s>p+\ve$ and $H_0(s)=h(\ve)$ for $s\le p+\ve$. Evidently, $H(s)\le H_0(s)$ for all $s.$
Therefore
\begin{align*}
\gIl{\c,A}{H(f)}&\le \gIl{\c,A}{H_0(f)}
=\sup _{0\le t\le h(\ve)}\{t\c \mu (A)\}\vee \sup_{t>h(\ve)} 
\left\{ t\c \mu(A\cap \{ H_0(f)\ge t\})\right\}
\\&= \big[h(\ve)\c \mu(A) \big]\vee \big[ \sup H\c \mu(A\cap \{f>p+\ve\})\big].
\end{align*}
As $\mu(A\cap\{f>y\})\le p$ for $y >p$, we  get, from the right-continuity of  $x\mapsto x\c p$, that
\begin{align}\label{a1}
\gIl{\c,A}{H(f)}\le \big[\big(H(p^+)\vee \sup H([0,p])\big)\c \mu(A)\big] \vee \big[\sup H\c p\big].
\end{align}
 The equality holds in \eqref{a1}, if $H$ is right-continuous at $p,$ $H(p)\ge \sup H([0,p]),$ 
and $f=y_0\mI{A},$ where $y_0$ is such that $H(y_0)=\sup H.$ 
\medskip

(ii) As the bound \eqref{in4} is obvious for $p=0$, we assume that $p>0.$ Put  $h(\ve)=\sup H([p-\ve,\infty])$ for $\ve\in (0,p)$.   
Set  $H_0(s)=\sup H$ for $s<p-\ve$ and $H_0(s)=h(\ve)$ for $s\ge p-\ve$. Since  $H(s)\le H_0(s)$ for all $s$, we get 
\begin{align*}
\gIl{\c,A}{H(f)}&\le \gIl{\c,A}{H_0(f)}=
\big[ h(\ve)\c \mu(A) \big]\vee \big[ \sup H\c \mu(A\cap \{f<p-\ve\})\big].
\end{align*}
Clearly, $\mu(A)\ge \mu(A\cap \{f<p-\ve\})+\mu(A\cap \{f\ge p-\ve\})$
and $\mu(A\cap \{f\ge p-\ve\})\ge p$, so
\begin{align*}
\gIl{\c,A}{H(f)}&\le  \big[h(\ve) \c \mu(A) \big]\vee \big[\sup H\c (\mu (A)-p)\big].
\end{align*}
Taking the limit as $\ve\to 0$, we obtain \eqref{in4} with equality if $H$ is left-continuous at $p$, 
$H(p)\ge  \sup H([p,\infty])$ and   
$f=\mu(A)\mI{A}$. 
\end{proof}

\begin{remark} The bound \eqref{in3} (resp. \eqref{in4})  is sharp for each $p,$
if $H$  is a~nondecreasing right-continuous function (resp. nonincreasing left-continuous function). Given  a~superadditive  monotone measure $\mu$ and a~continuous quasiconcave function $H,$ Theorem~\ref{tw2} implies that 
\begin{align}\label{in3a}
\gIl{\c,A}{H(f)}\le &\big[\sup H([0,p])\c \mu(A)\big]\vee \big[\sup H([p,\infty])\c \mu(A)\big]\nonumber\\
&\vee \big[\sup H\c p\big] \vee \big[\sup H\c (\mu(A)-p)\big]
\end{align} 
and  the bound \eqref{in3a} is sharp for every $p.$
\end{remark}

The following result is an immediate consequence of Theorem \ref{tw2}. 

\begin{corollary} \label{co2} 
Assume that a~continuous function $H$ is increasing on $[0,c]$ and decreasing on $[c,\infty]$, where $c\in [a,b]\subset \mR_+.$ If  $f\in \cF_{(X,[a,b])}$ and $p=\rI{A}{f}\le c$, then 
\begin{align*}
\rIl{A}{H(f)}\le (H(p)\vee p)\wedge H(c)\wedge \mu(A).
\end{align*}	
	  Moreover, if $c<p<\infty$ and $\mu$ is a~weakly superadditive monotone measure on $A,$ then 
	  \begin{align*}
	  \rIl{A}{H(f)}\le\big(H(p)\vee (\mu(A)-p)\big)\wedge H(c)\wedge \mu(A).
	  \end{align*}
\end{corollary}
\begin{proof}
Recall that $\rI{A}{f}\le \mu(A)$. Apply Theorem 2.2 with $\circ =\wedge$ and observe that $(H(p)\wedge \mu(A))\vee(H(c)\wedge p)=(H(p)\vee p)\wedge H(c)\wedge \mu(A)$ for $p\le c$ and $(H(p)\wedge \mu(A))\vee (H(c)\wedge (\mu(A)-p))=(H(p)\vee (\mu(A)-p))\wedge H(c)\wedge \mu(A)$ for $p>c.$
\end{proof}

As some nondecreasing binary maps $\c$ are not left-continuous (see e.g. \cite[Example 1.24]{klement2}), we provide 
modifications  of Theorems  \ref{tw1} and \ref{tw2}, which hold true without any continuity assumption on $\c.$  
Let us  recall that  a~monotone measure $\mu$ is {\it continuous from below} (resp. {\it from above}) if $\lim_{n\to \infty}\mu(A_n)=\mu(\lim_{n\to\infty}A_n)$ for all $A_n\in\cA$ such that $A_n\subset A_{n+1}$ (resp. $A_{n+1}\subset A_{n}$) for $n\in\mN.$  We say that $\mu$ is {\it continuous}, if it is both continuous from below and from above.  The following result generalizes Theorem~1  in
%of
 \cite{flores3}.

\begin{tw}\label{tw3}Let
$H\colon \moR_+\to \moR_+$ and
 $\c\colon \moR_+\times \moR_+ \to \moR_+$ be a~nondecreasing map such that $a\c 0=0$ for all $a,$ $f\in\cF_{(X,\moR_+)}$ and 
 $p=\rI{A}{f}<\infty$. Let  $\mu$ be a~continuous monotone measure on $X.$
\begin{enumerate}
\item[(i)] The following inequalities hold true
\begin{align}\label{ss1}
\gIl{\c,A}{H(f)}&\ge \big[\inf H([p,\infty])\c p\big]\vee \big[\inf H\c \mu(A)\big],\\
\gIl{\c,A}{H(f)}&\le \big[\sup H([0,p])\c \mu(A)\big]\vee \big[\sup H \c p\big]\label{ss2}.
\end{align}
There is equality in  \eqref{ss1}  
for $f=\mu(A)\mI{A}$  
if $H(p)=\inf H([p,\infty]).$ 
Equality holds in \eqref{ss2} if $f=y_0\mI{A}$, 
$H(p)=\sup H([0,p])$ and $H(y_0)=\sup H$ for some $y_0$.
\item[(ii)] If $\mu$ is weakly subadditive on $A$, then 
\begin{align}\label{ss3}
\gIl{\c,A}{H(f)}\ge   \big[\inf H([0,p]) \c (\mu(A)-p)\big]\vee \mmm.
\end{align}
 The bound \eqref{ss3} is reached  by  $f=y_0\mI{A}$ if  $H(p)=\inf H([0,p])$ and $H(y_0)=\inf H$ for some $y_0$.
\item[(iii)]  If  $\mu$ is weakly superadditive  on $A$, then
\begin{align}\label{ss4}
\gIl{\c,A}{H(f)}\le \big[\sup H([p,\infty])\c \mu(A)\big]\vee \big[\sup H \c (\mu (A)-p)\big].
\end{align}
The equality is attained in \eqref{ss4}  for  $f=\mu(A)\mI{A}$ if $H(p)=\sup  H([p,\infty])$.
\end{enumerate}
\end{tw}
\begin{proof} The proof of Theorem \ref{tw3} is 
%extremely 
similar to those of Theorem \ref{tw1} and \ref{tw2}; just put $\ve=0$ and use the fact that  if $\mu$  is a~continuous monotone measure, then  
$\mu(A\cap\{f\ge p\})\ge p$ and $\mu(A\cap\{f>p\})\le p.$ 
The last statement follows easily from the bounds 
$\mu(A\cap\{f\ge y\} )\ge p$ for $y<p$ and
$\mu(A\cap\{f\ge y\} )\le p$ for $y>p$ (see also \cite[Lemma 9.5]{wang}).
\end{proof} 
 
   Note that the bounds
%Bounds 
in Theorem~\ref{tw3} may be better than  their counterparts in  Theorems \ref{tw1} and \ref{tw2}. 
\begin{example}
Let $\mu$ be the Lebesgue measure on $X=[0,2]$. If $f(x)=x$ and $H(x)=0.5\mI{\{1\}}(x)+x^2\mI{(1,2]}(x),$ 
then $p=\rI{X}{f}=1$ and $\rI{X}{H(f)}=1$.
Inequality \eqref{in1} gives us the trivial bound  $\rI{X}{H(f)}\ge 0$, as $H(1_-)=0,$ while 
%we get 
from \eqref{ss1} 
we get
%that
$\rI{X}{H(f)}\ge 0.5.$
\end{example}

\begin{remark} 
 Corollary 3.6 of \cite{abba}  gives the upper bound for the Sugeno integral of a concave function, but  
the following counterexample shows that 
the result 
is false if $m>0$, where $m\in \partial \varphi(p)$. 
We follow the notation of \cite{abba}. Let $X=A=[0,1]$ and $\mu$ be the Lebesgue measure. 
Take $\varphi(x)=\sqrt{x}$ and $f(x)=0.5x.$  Then
$f(X)=[0,0.5],$ $p=1/3$ and $m=\varphi'(p)=0.5\sqrt{3}>0.$ By Corollary 3.6 in
 \cite{abba} we get 
\begin{align}\label{nn1}
\rIl{A}{\varphi(f)}\le \frac{m}{m+1} (0.5-p)+\frac{1}{m+1}\varphi(p).
\end{align}
An easy computation shows that $\rI{A}{\varphi(f)}=0.5$ and
the right-hand side of \eqref{nn1} is approximately equal to $0.39,$
so inequality \eqref{nn1} is invalid.
\end{remark}

\section{Jensen inequalities for nonnegative concave functions}

In this section we give some Liapunov type  inequalities, that is, we evaluate the integral
$\gI{\c,A}{H(f)}$ by means of  integrals $\gI{\c,A}{G(f)}$ and $\gI{\bullet,A}{f}.$  As a~consequence, we  obtain 
some new  Jensen type inequalities for nonnegative concave  functions.

\begin{tw}\label{tw4}
Let $\circ,\bullet \colon \mR_+\times\mR_+\to \mR_+$ be nondecreasing maps  such that $a\c 0=a\bullet 0=0$ for all $a$.
Let $H\colon \mR_+\to \mR_+$, $\mu\in\cM_{(X,\cA)},$ $A\in\cA,$ $f\in\cF_{(X,\mR_+)}$ and $p=\gI{\bullet,A}{f}<\infty.$
Assume that $(a+b)\c c\le (a\c c)+(b\c c)$ for all $a,b,c$. 
If 
there exists $m_p\in\mR$ such that 
$H(y)\le H(p)+m_p(y-p)$ for $y\ge 0,$   
then the following attainable bound holds  true
\begin{align}\label{in8}
\gIl{\c,A}{H(f)}\le \inf_{c\in \mR}\Big\{\big[\big(H(p)+m_p(c-p)\big)^+\c \mu(A)\big]+\gIl{\c,A}{(m_p(f-c))^+}\Big\},
\end{align}
where $a^+=a\vee 0.$
\end{tw}
\begin{proof}
By the assumption on $H$, we obtain  
\begin{align}\label{doda1}
\gIl{\c,A}{H(f)}&\le \gIl{\c,A}{\big(H(p)+m_p(c-p)+m_p(f-c)\big)}\nonumber\\
&\le \gIl{\c,A}{\big[\big(H(p)+m_p(c-p)\big)^++\big(m_p(f-c)\big)^+\big]},
\end{align}
where $c\in \mR.$ 
It is easy to check that  the generalized Sugeno integral has the scale translation property,  i.e.,
% if $(a+b)\c c\le (a\c c)+(b\c c)$ for $a,b,c\ge 0$, that is, for $a\ge 0$,
\begin{align}\label{doda2}
\gIl{\c,A}{(a+f)}\le (a\c \mu(A))+\gIl{\c,A}{f}
\end{align}
 for all $a\ge 0$ under the condition $(x+y)\c z\le (x\c z)+(y\c z)$ for all $x,y,z\ge 0$
(see  \cite{bo}). 
Inequality \eqref{in8} follows from \eqref{doda1} and \eqref{doda2}.
Bound \eqref{in8} is reached by the function $f=(a\bullet\mu (A))\mI{A},$  
where $a\ge 0$, if $\mu(A)\bullet \mu (A)=\mu(A)$ and the map  
$\bullet$ is associative. This follows from \eqref{in8}  
applied to $c=a\bullet\mu(A).$
\end{proof}

Denote by $\partial H(x)$ the subdifferential of a~concave function $H$ at point $x$ (see~\cite{abba}).  

\begin{corollary}\label{co3} Let $\mu\in\cM_{(X,\cA)},$  $H\colon \mR_+\to \mR_+$ be   a~concave  function and 
$m_p\in \partial H(p),$ where  $p=\rI{A}{f}<\infty$ and $f\in\cF_{(X,\mR_+)}.$ Then  
\begin{align}\label{in99}
\rIl{A}{H(f)}\le \big[H(p)\wedge \mu(A)\big]+\rIl{A}{\big(m_p(f-p)\big)^+}.
\end{align}
\end{corollary}
\begin{proof}
Put   $c=p$ and $\c=\bullet=\wedge$ in \eqref{in8}.
\end{proof}

Corollary~\ref{co3} shows that Theorem~\ref{tw4}  is a~generalization of Theorem 4.3 
 in
%of
~\cite{abba}. 
Indeed, bound \eqref{in99} was given in  \cite[Theorem 4.3]{abba}  under the assumption that   $H$ is an increasing concave function, $A=\{x_1,x_2,\ldots,x_n\}$ and $f(x_1)\ge f(x_2)\ge \ldots \ge f(x_n)$. Note that  for $m_p\ge 0$ we get
\begin{align*}
\rIl{A}{\big(m_p(f-p)\big)^+}&=\sup_{t\ge 0}\left\{t\wedge \mu(A\cap \{ m_p(f-p)^+\ge t\})\right\}\\
&=\max_{i}\left\{\big(m_p(f(x_i)-p)^+\big)\wedge \mu_i\right\}=\max_{i}\left\{\big(m_p(f(x_i)-p)\big)\wedge \mu_i\right\},
\end{align*}
where  $\mu_i=\mu(\{x_1,\ldots,x_i\})$.

\medskip
By \eqref{in8} we also obtain the following inequality for the Sugeno integral 
\begin{align}\label{l1}
\rIl{A}{H(f)}\le \big[(H(p)-pm_p)^+\wedge \mu(A)\big]+\rIl{A}{(m_p)^+f},
\end{align}
where $p=\rI{A}{f}$ and $f\in \cF_{(X,\mR_+)}$. Further, if
 $0<m_p\le 1$ and $f\in\cF_{(X,[0,1])},$ then  
combining the fact that  $m_py\le m_p\wedge y$ for $m_p,y\in [0,1]$  
with 
comonotone minitivity of the Sugeno integral, we obtain the  Jensen type bound of the form
\begin{align*}
\rIl{A}{H(f)}\le \big[(H(p)-pm_p)^+\wedge \mu(A)\big]+(m_p\wedge p).  
\end{align*}

The following example shows that the infimum in \eqref{in8} can be  achieved  at $c\notin \{0,p\}$.

\begin{example} 
Let $X=\mR$, $A=[0,5]$ and $\mu$ be 
the Lebesgue measure. 
Take $H(x)=\sqrt{x}$ and $f(x)=x.$ Then $p=\rI{A}{f}=2.5$ and  
\begin{align*}
\rIl[\mu]{A}{H(f)}=\frac{-1+\sqrt{21}}{2}\approx 1.7913.
\end{align*}
Write $g(c)=\big[\big(H(p)+m_p(c-p)\big)^+\wedge \mu(A)\big]+\rI{A}{\big(m_p(f-c)\big)^+}.$ Clearly,  
\begin{align*}
g(c)=\big[\big(\sqrt{2{.}5}+m_p(c-2{.}5) \big)^+ \wedge 5\big]+\Big[\big((-m_pc)^+\wedge 5\big)\vee \Big(\frac{m_p}{m_p+1}(5-c)^+\Big)\Big],
\end{align*}
where $m_p=H'(p)=1/\sqrt{10}.$ 
After an easy calculation we get
$\inf_{c\in \mR}g(c)=g(-2{.}5)\approx 1{.}8019,$ 
so 
the difference between the upper bound \eqref{in8} and the exact value of 
integral $\rI{A}{H(f)}$ is about $0{.}0106.$  
\end{example}
\medskip

We also give a~Jensen type inequality for the Shilkret integral. 

\begin{corollary} 
Let   $H\colon \mR_+\to \mR_+$ be a~differentiable and concave function.
% from $\mR_+$ to $\mR_+.$ 
Then for all  $f\in\cF_{(X,\mR_+)}$ and $\mu\in\cM_{(X,\cA)}$ we get
\begin{align}\label{in80}
\gIl{\cdot,A}{H(f)}\le H(p)\mu(A)+\big[(H'(p))^+-H'(p)\mu(A)\big]p,
\end{align}
where $p=\gI{\cdot,A}{f}<\infty.$ In particular, if $\mu(A)=1$ and $H'(p)\ge 0,$ then 
\begin{align*}
\gIl{\cdot,A}{H(f)}\le H\Big( \gIl{\cdot,A}{f}\Big).
\end{align*}  
\end{corollary}
\begin{proof} Take $c=0,$ $m_p=H'(p)$ and $\c=\bullet=\cdot$ 
in Theorem \ref{tw4}. Observe that $
 0\le H(0)\le H(p)-H'(p)p.$
\end{proof}

\section{Jensen type bounds for real-valued  functions}

Let  $\star\colon \mR_+\times \mR_- \to \mR$ be a~nondecreasing map, where $\mR=(-\infty,\infty)$ and $\mR_-=(-\infty,0].$ Suppose that  $f\in\mathcal{F}_{(X,\mR)}$ and write  $f^+=f\vee 0$ and $f^-=(-f)\vee 0.$ 
We define 
the \textit{$\star$-symmetric Sugeno integral} of $f$ on $A\in \cA$ by the formula 
\begin{align}\label{g1}
\rrI{f}:=\Big(\rIl{A}{f^+}\Big)\star \Big(-\,\rIl{A}{f^-}\Big),
\end{align}
provided that $\rI{A}{f^+}<\infty$ and $\rI{A}{f^-}<\infty.$
Kawabe \cite{ka} 
examined properties of the $+$-symmetric Sugeno integral while Grabisch proposed to use the {\it symmetric Sugeno integral} defined by \eqref{g1} with the operator 
$a\ovee b=\textrm{sign}(a+b)\big(|a|\vee |b|\big),$ where $a,b\in\mR$ (see~\cite{gra0,gra1}).   
       
We derive both lower and upper bound 
on the $\star$-symmetric Sugeno integral of  $H(f)$ by means of  the  Sugeno integrals $p_1:=\rI{A}{f^+}$ and $p_2:=\rI{A}{f^-},$ where $H\colon \mR \to \mR$ is a~nondecreasing function such that  $H(0)=0$.  
By the assumption on $H$, we have  $H(f(x))\vee 0=H_1(f^+(x))$ 
and 
$\big(-H(f(x))\big)\vee 0=H_2(f^-(x))$ for all $x\in X,$  where $H_1(x)=H(x)$ and $H_2(x)=-H(-x)$ for $x\ge 0$. Of course,  functions
$H_1,H_2\colon \mR_+\to \mR_+$ are nondecreasing, $H_1(0)=H_2(0)=0$ and 
\begin{align}\label{o0}
\rrI{H(f)}=\Big(\rIl{A}{H_1(f^+)}\Big)\star \Big(-\rIl{A}{H_2(f^-)}\Big).
\end{align}
Since  $\st$ is a~nondecreasing binary map,  we can apply Theorems
\ref{tw1}-\ref{tw3} to obtain two-sided bounds on $\rrI{H(f)}$.
Below, 
we provide  
the upper bound.  
Assume, for simplicity of exposition, that $\mu\in \cM_{(X,\cA)}$ is continuous. 
Further, assume that $p_1,p_2<\infty$ and $p_1\le \sup H.$  
Putting  $\c=\wedge$ in \eqref{ss1} and \eqref{ss2} 
we get  
%that 
\begin{align}\label{o2}
\rIl{A}{H_2(f^-)}&\ge H_2(p_2)\wedge p_2=\big(-H(-p_2)\big)
\wedge p_2,\\
\label{o3}
\rIl{A}{H_1(f^+)}&\le \big(H(p_1)\wedge \mu (A)\big)\vee (\sup H\wedge p_1)=\big(H(p_1)\vee p_1\big)\wedge \mu(A),
\end{align}  
because $p_1\le \mu(A)$. 
As a~consequence of \eqref{o0}-\eqref{o3},  we obtain the following bound  
\begin{align}\label{001}
\rrI{H(f)}\le \big[\big(H(p_1)\vee p_1\big)\wedge \mu(A)\big]\star \big[H(-p_2)\vee (-p_2)\big].
\end{align}  
The equality is reached in \eqref{001} for $f(x)=\mu(B)\mI{B}(x)-\mu(A\backslash B)\mI{A\backslash B}(x),$
where  $B\subset A$ is  such that $\mu(B),\mu(A\backslash B)<\infty$ and $H(\mu(B))=\mu(B)$.
Summing up, we arrive at the following result. 
\begin{tw}
Let $H\colon \mR\to \mR$ be a~nondecreasing function such that $H(0)=0.$ 
Then the sharp inequality \eqref{001} holds true for any continuous monotone measure $\mu$ and for all $f\in\cF_{(X,\mR_+)}$ such that the integrals
 $\rI{A}{f^+}$, $\rI{A}{f^-}$ are finite and  
 $\rI{A}{f^+}\le \sup H.$ 
\end{tw}

\begin{example}
Let $X=A=\{1,2,3\}.$ Suppose that
\begin{eqnarray*}
\mu(\{1\})=0.1,\quad &\mu(\{1,2\})=0.4,\quad \quad& f(1)=-1,\\
\mu(\{2\})=0.25,\quad &\mu(\{1,3\})=0.3,\quad \quad& f(2)=0.3,\\
\mu(\{3\})=0.2,\quad &\mu(\{2,3\})=0.6,\quad \quad& f(3)=1
\end{eqnarray*}  
and $\mu(\{1,2,3\})=1.$
If $H(x)=x^3,$ then
\begin{align*}
p_1=0.3,\quad p_2&=0.1,\quad 
\rIl{A}{H_1(f^+)}=0.2,\quad \rIl{A}{H_2(f^-)}=0.1.
\end{align*}
Hence, we get from \eqref{001} that 
\begin{align*}
0.1=0.2+(-0.1)&=\mathrm{Su}^{+}_{\mu,A}(H(f))\le p_1+(-p_2)=0.2,\\
0.2=0.2\ovee (-0.1)&=\mathrm{Su}^{\ovee}_{\mu,A}(H(f))\le 
p_1\ovee (-p_2)=0.3.
\end{align*}
\end{example}

\begin{example}
Assume that  $X=\mR$, $A=[-3,1]$ and $\mu=\sqrt{\lambda},$ where $\lambda$ is the Lebesgue measure. Put  $f(x)=x$ and $H(x)=x\mI{\mR_+}(x)+2x\mI{\mR_-}(x).$ Then
\begin{align*}
p_1=\frac{\sqrt{5}-1}{2},\quad p_2&=\frac{\sqrt{13}-1}{2},\quad 
\rIl{A}{H_1(f^+)}=\frac{\sqrt{5}-1}{2},\quad \rIl{A}{H_2(f^-)}=1.5.
\end{align*}
It follows from \eqref{001} that \begin{align*}
\frac{\sqrt{5}-4}{2}&=\mathrm{Su}^{+}_{\mu,A}(H(f))\le
p_1-p_2=\frac{\sqrt{5}-\sqrt{13}}{2},\\
-1.5&=\mathrm{Su}^{\ovee}_{\mu,A}(H(f))\le p_1
\ovee (- p_2)
=\frac{1-\sqrt{13}}{2}\approx -1.3.
\end{align*}
\end{example}
\bigskip

Similar result as in \eqref{001} can be obtained  provided that  $H$ is nonincreasing and $H(0)=0$.
If $\mu$ is subadditive and  $H\colon \mR\to \mR_+$ is nonincreasing for $x\le 0,$ nondecreasing for $x\ge 0$ and $H(0)=0,$ then 
\begin{align*}
\rIl{A}{H(f)}&\le \sup_{t\ge 0}\left\{t\wedge \mu\big(A\cap \lbrace{f\ge 0}\rbrace\cap \lbrace H(f)\ge t\rbrace \big)\right\}\\
&\quad+\sup_{t\ge 0}\left\{t\wedge \mu\big(A\cap \lbrace{-f\ge  0}\rbrace\cap \lbrace \widetilde{H}(-f)\ge t\rbrace \big)\right\}\\
&=\rIl{A}{H(f^+)}+\rIl{A}{\widetilde{H}(f^-)},
\end{align*}
where $\widetilde{H}(x)=H(-x)$ for $x\ge 0$. Thus,  
 the
%an 
upper bound can be derived from \eqref{in3} or  \eqref{ss2}. Clearly, for 
any $\mu\in\cM_{(X,\cA)},$
\begin{align*}
\rIl{A}{H(f)}&\ge \rIl{A}{H(f^+)} \vee \rIl{A}{\widetilde{H}(f^-)},
\end{align*}
so we can also give a~lower bound on $\rI{A}{H(f)}.$ 
Further, in a~similar way as above, we can also estimate the {\it $(\star,\c)$-asymmetric integral} defined by 
\begin{align*}
\mathrm{Su}^{\star,\c}_{\mu,\nu,A}(H(f)):=\Big(\gIl{\c,A}{H(f)\vee 0}\Big)\star
\Big(-\gIl[\nu]{\c,A}{(-H(f))\vee 0}\Big),
\end{align*}
where $\mu,\nu\in\cM_{(X,\cA)}.$
See e.g. \cite{pap} for 
the motivation of this 
definition with $\circ=\wedge$ and $\st=\ovee.$

\section{Conclusions}

In this paper, we have provided optimal lower/upper bounds of the Jensen type for the generalized Sugeno integral of  measurable real-valued functions. As a consequence, we have obtained the Jensen type inequalities  for the Sugeno integral,  Shilkret integral and q-integral.    
Our results generalize and improve a number of  
known results. 

The Jensen type inequalities for fuzzy integrals can be a~useful tool to solve both theoretical and practical problems in many areas of research as
the concept of the  Sugeno integral has numerous applications.  
%e.g.Jensen type operators to the theory of risk aversion can be found in \cite{kaluszka}.  
 The Sugeno integral plays important role  in
decision-making problems under uncertainty and multi-criteria decision
problems \cite{du}. 
The famous Hirsch index \cite{hirsch}, which is closely related to the 
%discrete 
Sugeno integral \cite{TN}, is widely used in evaluation of research performance of individual scientists, research groups and universities.
Nurukawa and Torra \cite{Nur2} described the use of the Sugeno integral in decision making when modeling auctions. 
An application 
%of Jensen type operators to the theory 
of risk 
%aversion
theory can be found in \cite{kaluszka}.  
The Sugeno integral was applied to describe a~face recognition using modular neural networks with a~fuzzy logic method  \cite{melin}. 
Hu \cite{Hu} proposed a fuzzy data mining method with the Sugeno fuzzy integral that can effectively find a~compact set of fuzzy if-then classification rules.
Some applications to fuzzy inference systems were given in \cite{Nur1}. 
For more details about
possible applications of the  Sugeno integral, we refer to \cite{abba,gra2,gra,gra1,wang,Ying}.
\medskip
\medskip

{\bf Acknowledgements}
\medskip

The authors would like to thank the referees for their  comments which led to improvements in the paper.

\addcontentsline{toc}{chapter}{Literatura}
\end{document}